\documentclass[a4paper, reqno, 11pt]{amsart}
\makeatletter

\@addtoreset{equation}{section}
\makeatother
\usepackage{setspace}
\usepackage{xcolor}
\usepackage{braket}
\usepackage[T1]{fontenc}
\everymath{\displaystyle}

\usepackage{amsmath}
\usepackage{amssymb}
\usepackage{latexsym}
\usepackage{amsthm}
\usepackage{textcomp}
\usepackage{hyperref}
\usepackage{mathtools}
\usepackage{here}
\newtheorem{Thm}{Theorem}[section]

\newtheorem{Lem}[Thm]{Lemma}
\newtheorem{Prop}[Thm]{Proposition}

\theoremstyle{definition}

\newtheorem{Rem}[Thm]{Remark}

\begin{document}

\title[multiple point ranges on graphs]{Asymptotic estimates for multiple point ranges of transient random walks on graphs}
\author{Kazuki Okamura}
\date{\today}
\address{Department of Mathematics, Faculty of Science, Shizuoka University, 836, Ohya, Suruga-ku, Shizuoka, 422-8529, JAPAN.}
\email{okamura.kazuki@shizuoka.ac.jp}
\subjclass[2020]{}
\maketitle

\begin{abstract}
We study multiple point ranges for random walks on graphs, extending known asymptotic results obtained for random walks on groups.
A distinctive feature is that algebraic and translation-invariance assumptions are replaced by a uniform tail condition on the first return time to the starting point. 
Under this condition, we obtain upper and lower bounds of linear order for the expectations of the number of sites visited at least a given number of times and the number of sites visited exactly that number of times. 
We also prove the corresponding almost sure bounds under a stronger condition.  
In spatially homogeneous transient cases, these bounds coincide and yield a strong law of large numbers.
We apply these estimates to derive asymptotic results for functions of the local times.
\end{abstract}

\section{Introduction and main results}

The range of a random walk is one of the basic quantities measuring how much of the state space the walk has explored. 
It counts sites without distinguishing how many times they have been visited. 
In order to describe the occupation structure more finely, we count sites according to their number of visits. 
For each $j \ge 1$, we consider the number $R_n^j$ of sites visited at least $j$ times by time $n$, and the number $R_n^{(j)}$ of sites visited exactly $j$ times by time $n$. 
Here $R_n^{1}$ is the single point range. 
These quantities are closely related to self-intersections and to the local times of the walk.

The asymptotic behavior for the multiple point ranges $R_n^{j}$ and $R_n^{(j)}$  has been investigated by several authors. 
For transient random walks, Erd\H{o}s and Taylor \cite{ET1960} showed the strong law of large numbers for the simple random walk on the integer lattice $\mathbb{Z}^d, d \ge 3$. 
Pitt \cite{Pitt1974}  showed the strong law of large numbers for a transient random walk on a countable Abelian group. 
Derriennic \cite{Derriennic1980} considered a more general class of random walks including the case of random walks on countable non-Abelian groups.
Multiple point ranges for the two-dimensional random walk have also been considered by  Flatto \cite{Flatto1976} and Hamana \cite{hamana1997,Hamana1998}. 
These works rely on the homogeneous structure of the underlying space. 
In particular, group structures make it possible to use ergodic methods. 
For random walks on general graphs, such tools do not apply. 
This motivates the use of assumptions formulated in terms of return probabilities, rather than algebraic structure. 

The aim of this paper is to obtain analogues of known asymptotic results for multiple point ranges for random walks on graphs.  
Instead of relying on translation invariance or ergodic arguments, we assume a uniform tail condition for the first return time to the starting point. 
Under this condition, we prove upper and lower bounds of linear order for the expectations of $R_n^j$ and $R_n^{(j)}$.
Furthermore, under a stronger uniform tail condition, we obtain the corresponding almost sure bounds. 
If the probability of returning to the starting point is independent of the starting point and is strictly less than one, 
then the upper and lower bounds coincide and a strong law of large numbers holds. 

The case of sites visited at least once corresponds to the single point range. 
Our results also recover estimates for the single point range, and the uniform tail condition here is weaker than the one used by the author  \cite[Corollary 2.3]{Okamura2014} and by Kumagai and Nakamura \cite[Proposition 5.7]{KN2018}. 
As discussed in Section \ref{sec:examples}, this condition can be verified through on-diagonal heat kernel upper bounds. 
It is therefore applicable to a broad class of transient random walks on graphs.

Multiple point ranges are closely related to functions of the local times. 
For transient random walks on lattices and groups, asymptotic properties of local time functionals have been studied by several authors. 
Becker and K\"{o}nig \cite{BK2009} investigated moments of the local times of transient random walks on $\mathbb Z^d$. 
Asymont and Korshunov \cite{AK2020} considered a more general class of functions of the local times.
Recently, Chang, Chen, Meng and Peng \cite{CCMP2026}  obtained further limit results for such functionals.
These results are proved for integer lattices and groups.

We consider the corresponding problem for random walks on graphs in  Section \ref{sec:local-times}. 
For a non-negative function $f$ with $f(0)=0$, let $G_n(f)$ be the sum obtained by applying  $f$ to the local time of each site visited up to time $n$.
Combining a decomposition by local time levels with our estimates for exact multiple point ranges $R_n^{(j)}$, we obtain asymptotic results for $G_n(f)$ under suitable summability conditions on $f$. 
The local time results are graph analogues of known results in the lattice and group settings and  include an extension of \cite[Theorem 2]{CCMP2026} and of the $L^2$-convergence part of \cite[Theorem 1]{AK2020}.

\subsection{Framework and main results}

Let $\mathcal{X}$ be a countable set. 
Let $(X_n)_{n=0}^{\infty}$ be an irreducible Markov chain on $\mathcal{X}$. 
If $P(X_0 = x) = 1$, then we denote the law of the Markov chain $(X_n)_n$ by $P^x$. 
We denote the expectation with respect to $P^x$ by $E^x$. 

Let $\ell(n,x) \coloneqq |\{i \in \{1,\dots,n\} : X_i = x \}|$. 
For $j \ge 1$ and $n \ge 1$, let 
\[ R_n^{j} \coloneqq \left|\left\{x \in \mathcal{X} : \ell(n,x) \ge  j\right\}\right| \]
and 
\[ R_n^{(j)} \coloneqq \left|\left\{x \in \mathcal{X} : \ell(n,x) =  j\right\}\right|. \]
We remark that $R_n \coloneqq R_n^1$ is the range of the random walk and $R_{n}^{j} = R_{n}^{(j)} + R_n^{j+1}$. 

For $x \in \mathcal{X}$, let 
$T_{x} \coloneqq \inf\{n \ge 1 : X_n = x \}$. 
Let 
\[ F_{\inf} \coloneqq \inf_{x \in \mathcal{X}} P^x (T_x < \infty), \ \textup{ and } \   F_{\sup} \coloneqq \sup_{x \in \mathcal{X}} P^x (T_x < \infty).  \]

We introduce the following uniform tail condition for the first return time to the starting point. 
We say that $(U_0)$ holds if 
\[ \lim_{n \to \infty} \sup_{x \in \mathcal{X}} P_x (n < T_x < \infty) = 0, \]
and that $(U_1)$ holds if 
\[  \sup_{x \in \mathcal{X}} P_x (n < T_x < \infty) = O\left((\log n)^{-1-\delta} \right), \ n \to \infty. \]

\begin{Thm}\label{thm:main-exp}
If $(U_0)$ holds, then, for every $j \ge 1$, 
\begin{align*} 
(F_{\inf})^{j-1} (1-F_{\sup}) &\le  \liminf_{n \to \infty} \frac{\inf_{x \in \mathcal{X}} E^{x} \left[R_n^j \right]}{n} \\
&\le \limsup_{n \to \infty} \frac{\sup_{x \in \mathcal{X}} E^{x} \left[R_n^j \right]}{n}  \le (F_{\sup})^{j-1} (1-F_{\inf}) 
\end{align*}
and 
\begin{align*}  
(F_{\inf})^{j-1} (1-F_{\sup})^2 &\le  \liminf_{n \to \infty} \frac{\inf_{x \in \mathcal{X}} E^{x} \left[R_n^{(j)} \right]}{n} \\
&\le \limsup_{n \to \infty} \frac{\sup_{x \in \mathcal{X}} E^{x} \left[R_n^{(j)} \right]}{n}  \le (F_{\sup})^{j-1} (1-F_{\inf})^2. 
\end{align*}
\end{Thm}

If  the Markov chain $(X_n)_n$ is the simple random walk on a vertex-transitive graph or the random walk on a countable group, then $F_{\sup} = F_{\inf}$. 
If  the Markov chain $(X_n)$ is recurrent, then $F_{\sup} = F_{\inf} = 1$. 
$(U_0)$ is identical with the condition in \cite[Lemma 2.1]{Okamura2014}. 

\begin{Thm}\label{thm:main-SLLN}
If $(U_1)$ holds, then, for every $j \ge 1$ and every $x \in \mathcal{X}$, 
the following statements hold $P^x$-a.s.: 
\begin{equation}\label{eq:upper-SLLN-1}
\limsup_{n \to \infty} \frac{R_n^j}{n}  \le (F_{\sup})^{j-1} (1-F_{\inf}). \ 
\end{equation}
\begin{equation}\label{eq:upper-SLLN-2}
\limsup_{n \to \infty} \frac{R_n^{(j)}}{n}  \le (F_{\sup})^{j-1} (1-F_{\inf})^2. 
\end{equation} 
\begin{equation}\label{eq:lower-SLLN-1}
\liminf_{n \to \infty} \frac{R_n^j}{n}  \ge (F_{\inf})^{j-1} (1-F_{\sup}). \ 
\end{equation}
\begin{equation}\label{eq:lower-SLLN-2}
\liminf_{n \to \infty} \frac{R_n^{(j)}}{n}  \ge (F_{\inf})^{j-1} (1-F_{\sup})^2. 
\end{equation} 
\end{Thm}

In Section \ref{sec:examples}, we give a sufficient condition for $(U_1)$ in terms of heat kernel. 
To our knowledge, the Nash-type on-diagonal heat kernel upper bounds needed to verify $(U_1)$ are not known to hold for all transient random walks on countable groups.
However, under additional assumptions, Nash-type bounds and related estimates have been studied extensively; see Saloff-Coste and Zheng  \cite{SCZ2016} for example. 
We remark that there exist infinite vertex-transitive graphs which are not Cayley graphs of groups, as shown by Woess \cite{Woess2013}. 

\section{Proofs}

In this section, we prove Theorems \ref{thm:main-exp} and \ref{thm:main-SLLN}. 
As an outline, we follow \cite{Pitt1974}; in particular, we use induction on $j$ in the estimates of expectations. 
However, we cannot apply the ergodic theorem and several parts of the argument have to be changed significantly. 
The main difficulty is to estimate $E_{m,n}$ defined in \eqref{eq:Emn} below.
In \cite{Pitt1974},  this part is handled by the ergodic theorem; here we use the uniform tail condition $(U_1)$ instead. 

\subsection{Proof of Theorem \ref{thm:main-exp}} 

Let $\mathcal{A}_{k}^{n} \coloneqq \left\{X_k \ne X_{\ell}, k+1 \le \ell \le n \right\}$ for $k < n$ and $\mathcal{A}_{n}^{n}$ be the whole event. 
Let $\mathcal{B}_{k}^{(j)} \coloneqq \left\{   |\{i \in \{1,\dots,k\} : X_i = X_{k} \}| =  j \right\}$. 
We remark that $\mathcal{B}_{k}^{(j)} = \emptyset$ if $j > k$.  
Then, 
\[ R^{(j)}_{n} = \sum_{k=1}^{n} {\bf 1}_{\mathcal{B}_{k}^{(j)} \cap \mathcal{A}_{k}^{n}}, \textup{ and } R^{j}_{n} = \sum_{k=1}^{n} {\bf 1}_{\mathcal{B}_{k}^{(j)}}. \] 
Hence, for every $x \in \mathcal{X}$, 
\begin{equation}\label{eq:basic-exp-1} 
E^x \left[ R^{(j)}_{n} \right] = \sum_{k=1}^{n} P^x \left(\mathcal{B}_{k}^{(j)} \cap \mathcal{A}_{k}^{n}\right) \textup{ and } E^x \left[ R^{j}_{n} \right] = \sum_{k=1}^{n} P^x \left(\mathcal{B}_{k}^{(j)} \right). 
\end{equation}

For $n \ge 0$, let  
\[ F_{\inf} (n) \coloneqq \inf_{x \in \mathcal{X}} P^x (T_x \le n) \ \textup{ and } \   F_{\sup} (n) \coloneqq \sup_{x \in \mathcal{X}} P^x (T_x \le n).  \]

Then, by the Markov property, 
\begin{equation}\label{eq:basic-Markov-ineq} 
1 - F_{\sup}(n-k) \le P^x \left( \mathcal{A}_{k}^{n} \, \middle| \, \mathcal{B}_{k}^{(j)} \right) \le 1 - F_{\inf}(n-k), \ x \in \mathcal{X}. 
\end{equation}

First, we deal with the upper bounds. 
By \eqref{eq:basic-exp-1} and \eqref{eq:basic-Markov-ineq}, 
\[ E^x \left[ R^{(j)}_{n} \right] \ge \sum_{k=1}^{n} P^x \left(\mathcal{B}_{k}^{(j)} \right) (1- F_{\sup}(n-k)), \ x \in \mathcal{X}. \]
Using this, \eqref{eq:basic-exp-1}, and the fact that $F_{\sup}(m) \le F_{\sup}$ for each $m \ge 1$,   
it holds that 
\begin{equation}\label{eq:comparision-ineq-1} 
E^x \left[ R^{(j)}_{n} \right] \ge (1-F_{\sup}) E^x \left[ R^{j}_{n} \right], \ x \in \mathcal{X}. 
\end{equation}
Since $R_n^j - R_n^{j+1} = R_n^{(j)}$, 
\[ E^x \left[ R^{j+1}_{n} \right] \le F_{\sup} E^x \left[ R^{j}_{n} \right], \ x \in \mathcal{X}, \]
and hence, 
\begin{equation}\label{eq:exp-unif-upperbound} 
\sup_{x \in \mathcal{X}} E^x \left[ R^{j+1}_{n} \right] \le F_{\sup}  \sup_{x \in \mathcal{X}} E^x \left[ R^{j}_{n} \right]. 
\end{equation}

We see that 
\[ \limsup_{n \to \infty} \frac{\sup_{x \in \mathcal{X}} E^{x} \left[R_n \right]}{n} \le 1 - F_{\inf}. \]
Hence, by induction on $j$, 
\begin{equation}\label{eq:upper-exp-1} 
\limsup_{n \to \infty} \frac{\sup_{x \in \mathcal{X}} E^{x} \left[R_n^j \right]}{n} \le (F_{\sup})^{j-1} (1 - F_{\inf}). 
\end{equation}

We consider $E^x \left[ R^{(j)}_{n} \right]$. 
By \eqref{eq:basic-exp-1} and \eqref{eq:basic-Markov-ineq}, 
\begin{equation}\label{eq:upper-exp-2}  
E^x \left[ R^{(j)}_{n} \right] \le  \sum_{k=1}^{n} P^x \left(\mathcal{B}_{k}^{(j)} \right) (1- F_{\inf}(n-k)),  \ x \in \mathcal{X}.  
\end{equation}
By $(U_0)$, it holds that 
\[ \lim_{n \to \infty} F_{\inf} (n) = F_{\inf} \ \textup{ and } \  \lim_{n \to \infty} F_{\sup} (n) = F_{\sup}. \]

Let $\epsilon > 0$. 
Let $N_{\epsilon}$ be an integer such that for every $n > N_{\epsilon}$, $F_{\inf}(n) \ge F_{\inf} - \epsilon$. 
Then, for every $n > N_{\epsilon}$ and $x \in \mathcal{X}$, 
\begin{equation*}
\sum_{k=1}^{n} P^x \left(\mathcal{B}_{k}^{(j)} \right) (1- F_{\inf}(n-k)) \le (1 - F_{\inf} + \epsilon) \sum_{k=1}^{n - N_{\epsilon}}  P^x \left(\mathcal{B}_{k}^{(j)} \right) + N_{\epsilon}. 
\end{equation*}
Combining this with \eqref{eq:upper-exp-2} and \eqref{eq:basic-exp-1}, 
\begin{equation}\label{eq:comparision-ineq-2-1} 
E^x \left[ R^{(j)}_{n} \right]  \le  (1 - F_{\inf} + \epsilon) E^x \left[ R^{j}_{n} \right]  + N_{\epsilon}, \ n > N_{\epsilon}, \ x \in \mathcal{X}.  
\end{equation}
By this and \eqref{eq:upper-exp-1},
\[ \limsup_{n \to \infty} \frac{\sup_{x \in \mathcal{X}} E^{x} \left[R_n^{(j)} \right]}{n}  \le (1 - F_{\inf} + \epsilon)  (F_{\sup})^{j-1} (1 - F_{\inf}).   \]
Letting $\epsilon \to +0$, 
\[ \limsup_{n \to \infty} \frac{\sup_{x \in \mathcal{X}} E^{x} \left[R_n^{(j)} \right]}{n}  \le (F_{\sup})^{j-1} (1 - F_{\inf})^2.   \]

We now deal with the lower bound. 
By \eqref{eq:comparision-ineq-2-1} and $R_n^j - R_n^{j+1} = R_n^{(j)}$, 
\[ (F_{\inf} - \epsilon) E^x \left[ R^{j}_{n} \right] \le E^x \left[ R^{j+1}_{n} \right] + N_{\epsilon}, \ n > N_{\epsilon}, \ x \in \mathcal{X}. \]
Hence, 
\[ (F_{\inf} - \epsilon) \inf_{x \in \mathcal{X}} E^x \left[ R^{j}_{n} \right] \le \inf_{x \in \mathcal{X}} E^x \left[ R^{j+1}_{n} \right] + N_{\epsilon}, \ n > N_{\epsilon}. \]
Dividing by $n$ and letting $n \to \infty$, 
\[ (F_{\inf} - \epsilon) \liminf_{n \to \infty} \frac{\inf_{x \in \mathcal{X}} E^{x} \left[R_n^j \right]}{n} \le \liminf_{n \to \infty} \frac{\inf_{x \in \mathcal{X}} E^{x} \left[R_n^{j+1} \right]}{n}. \]
Letting $\epsilon \to +0$, 
\[ F_{\inf} \liminf_{n \to \infty} \frac{\inf_{x \in \mathcal{X}} E^{x} \left[R_n^j \right]}{n} \le \liminf_{n \to \infty} \frac{\inf_{x \in \mathcal{X}} E^{x} \left[R_n^{j+1} \right]}{n}. \]

By using the last-exit decomposition as in the proof of  \cite[Theorem 1.2]{Okamura2014}, 
we obtain that
\begin{equation}\label{eq:exp-lower-most-basic} 
E^{x} \left[R_n \right] \ge \sum_{k=1}^{n} \inf_{y \in \mathcal{X}} P^y (T_y \ge k) \ge n \inf_{y \in \mathcal{X}} P^y (T_y = \infty) = n(1 - F_{\sup}). 
\end{equation}

It follows that 
\[ \liminf_{n \to \infty} \frac{\inf_{x \in \mathcal{X}} E^{x} \left[R_n \right]}{n} \ge 1 - F_{\sup}. \]
Therefore, by induction on $j$, 
\begin{equation*}\label{eq:lower-exp-1} 
\liminf_{n \to \infty} \frac{\inf_{x \in \mathcal{X}} E^{x} \left[R_n^j \right]}{n} \ge (F_{\inf})^{j-1} (1 - F_{\sup}). 
\end{equation*}
By this and \eqref{eq:comparision-ineq-1}, 
\[ \liminf_{n \to \infty} \frac{\inf_{x \in \mathcal{X}} E^{x} \left[R_n^{(j)} \right]}{n} \ge (F_{\inf})^{j-1} (1 - F_{\sup})^2. \]
This completes the proof. 

\subsection{Proof of Theorem \ref{thm:main-SLLN}}

We first deal with the upper estimates. 

Let 
\[ T^{j}_{k,n} \coloneqq \left|\left\{x \in \mathcal{X} \colon \ell(kn,x) - \ell((k-1)n,x) \ge  j \right\}\right|  \]
and 
\[ T^{(j)}_{k,n} \coloneqq \left|\left\{x \in \mathcal{X} \colon  \ell(kn,x) - \ell((k-1)n,x) =  j \right\}\right|.  \]
Let 
$\mathcal{R}_{k,n} \coloneqq \left\{X_{i} \colon (k-1)n + 1 \le i \le kn \right\}$  
and  
\begin{equation}\label{eq:Emn} 
E_{m,n} \coloneqq \left|\bigcup_{k_1, k_2 \in \{1, \dots, m\}, k_1 \ne k_2} \mathcal{R}_{k_1, n} \cap \mathcal{R}_{k_2, n} \right|. 
\end{equation}

Then, 
\[ R_{mn}^{j} \le \sum_{k=1}^{m}   T^{j}_{k,n} + E_{m,n}, \ \textup{ and }  \  R_{mn}^{(j)} \le \sum_{k=1}^{m}   T^{(j)}_{k,n} + E_{m,n}. \]
For notational convenience, let 
\[ F^{(u)}_{j} \coloneqq (F_{\sup})^{j-1} (1-F_{\inf}), \ \textup{ and }  \ F^{(u)}_{(j)} \coloneqq (F_{\sup})^{j-1} (1-F_{\inf})^2. \]
Then, for every $\epsilon > 0$, 
\[ P^x \left(R_{mn}^{j} \ge mn(F_j^{(u)} + \epsilon)\right) \le P^x \left(  \sum_{k=1}^{m}   T^{j}_{k,n} \ge mn \left(F_j^{(u)} + \frac{\epsilon}{2} \right) \right) + P^x \left( E_{m,n} \ge mn \frac{\epsilon}{2} \right). \]

We first give an upper bound for  $\displaystyle P^x \left(  \sum_{k=1}^{m}   T^{j}_{k,n} \ge mn \left(F_j^{(u)} + \frac{\epsilon}{2} \right) \right)$. 

\begin{Lem}\label{lem:Laplace-estimate-bdd}
For every $\epsilon > 0$, there exists $\theta(\epsilon) > 0$ such that 
for every random variable $Y$ on a probability space satisfying $|Y| \le 1$, 
\begin{equation}\label{eq:every-bdd-rv-Laplace}
E\left[\exp(\theta(\epsilon) Y) \right] \le 1 + \theta(\epsilon) (E[Y] + \epsilon) \le \exp\left(  \theta(\epsilon) (E[Y] + \epsilon) \right). 
\end{equation}
\end{Lem}

\begin{proof}
By the Lebesgue convergence theorem and the fact that $E[Y^n] \le 1$, 
\[  E\left[\exp(\theta(\epsilon) Y) \right] = \sum_{n=0}^{\infty} \frac{\theta^n}{n!} E[Y^n] \le 1 + \theta E[Y] + \frac{\theta^2}{2} \exp(\theta).    \]
Let $\theta(\epsilon)$ be a positive constant such that 
$\theta \exp(\theta) = \epsilon$. 
Then, using the inequality $1+x \le \exp(x)$, 
we have \eqref{eq:every-bdd-rv-Laplace}. 
\end{proof}

\begin{Lem}\label{lem:Laplace-estimate-multi-range}
For every $\epsilon > 0$, there exists  $N(\epsilon) \in \mathbb{N}$ such that 
for every $n \ge N(\epsilon)$, 
\begin{equation}\label{eq:multi-range-Laplace}
\sup_{x \in \mathcal{X}} E^x \left[\exp\left(\frac{\theta(\epsilon/2)}{n} R_n^j \right) \right] \le \exp\left(\theta(\epsilon/2) (F_j^{(u)} + \epsilon) \right),
\end{equation}
where $\theta (\cdot)$ is the function appearing in Lemma \ref{lem:Laplace-estimate-bdd}.  
\end{Lem}

\begin{proof}
Since $0 \le R^j_n \le n$, 
by Lemma \ref{lem:Laplace-estimate-bdd}, 
\[ \sup_{x \in \mathcal{X}} E^x \left[\exp\left(\frac{\theta(\epsilon/2)}{n} R_n^j \right) \right] \le \exp\left(  \theta(\epsilon/2) \left(\frac{E^x \left[R_n^j \right] }{n} + \frac{\epsilon}{2} \right) \right). \]
By Theorem \ref{thm:main-exp}, 
there exists $N(\epsilon) \in \mathbb{N}$ such that 
for every $n \ge N(\epsilon)$, 
$\displaystyle \frac{E^x \left[R_n^j \right] }{n} \le F_j^{(u)} + \frac{\epsilon}{2}$. 
Hence, 
\eqref{eq:multi-range-Laplace} holds for every $n \ge N(\epsilon)$. 
\end{proof}

\begin{Lem}\label{lem:upper-Markov-tail}
For every $\epsilon > 0$, there exist $C(\epsilon) > 0$  and $N_1 (\epsilon) \in \mathbb{N}$ such that 
for every $n \ge N_1 (\epsilon)$ and every $m \in \mathbb{N}$, 
\begin{equation*}\label{eq:multi-range-tail}
\sup_{x \in \mathcal{X}} P^x \left(  \sum_{k=1}^{m}   T^{j}_{k,n} \ge mn \left(F_j^{(u)} + \frac{\epsilon}{2} \right) \right)  \le \exp\left(-C(\epsilon) m \right). 
\end{equation*}
\end{Lem}

\begin{proof}
By the Markov property, for every $\lambda > 0$ and $x \in \mathcal{X}$, 
\[  P^x \left(  \sum_{k=1}^{m}   T^{j}_{k,n} \ge mn \left(F_j^{(u)} + \frac{\epsilon}{2} \right) \right) \le \exp\left(-\lambda mn \left(F_j^{(u)} + \frac{\epsilon}{2} \right) \right) E^x \left[ \prod_{k=1}^{m} \exp\left( \lambda T^{j}_{k,n}  \right)  \right] \]
\[ \le \left( \exp\left(-\lambda n \left(F_j^{(u)} + \frac{\epsilon}{2} \right) \right) \sup_{x \in \mathcal{X}} E^x \left[ \exp(\lambda R^j_n) \right] \right)^m. \]
Let $n \ge N(\epsilon/4)$ and $\lambda = \frac{1}{n} \theta(\epsilon/8)$. 
Then, by Lemma \ref{lem:Laplace-estimate-multi-range}, 
\[  \exp\left(-\lambda n \left(F_j^{(u)} + \frac{\epsilon}{2} \right) \right) \sup_{x \in \mathcal{X}} E^x \left[ \exp(\lambda R^j_n) \right] \le \exp\left(- \frac{\epsilon}{4} \theta(\epsilon/8)\right). \]
Thus the assertion holds for $\displaystyle C(\epsilon) \coloneqq \frac{\epsilon}{4} \theta(\epsilon/8)$ and $N_1 (\epsilon) \coloneqq N(\epsilon/4)$.
\end{proof}

Second, we give an upper bound for  $\displaystyle P^x \left( E_{m,n} \ge mn \frac{\epsilon}{2} \right)$. 
Let 
\[ \mathcal{A}_{k}^{\infty} \coloneqq \left\{X_{\ell} \ne X_k,  \ell \ge k+1 \right\} = \bigcap_{n \ge k+1} \mathcal{A}_{k}^n.\]  
Then, by the same argument as in \cite[Proof of Lemma]{Pitt1974}, 
it holds that for $n^{\prime} < n$, 
\[ E_{m,n} \le \sum_{k=1}^{mn} {\bf 1}_{\mathcal{A}_{k}^{k+n^{\prime}} \setminus \mathcal{A}_{k}^{\infty}} + mn^{\prime}. \]

Henceforth, we denote the integer part of a real number $x$ by $\lfloor x \rfloor$. 
For $n \ge 2$, let 
\begin{equation}\label{eq:def-n-prime}
n^{\prime} \coloneqq \left\lfloor \frac{n}{\log n} \right\rfloor. 
\end{equation}
Then, there exists $N_2 (\epsilon) \in \mathbb{N}$ such that for every $n \ge N_2 (\epsilon)$, $n^{\prime}/n < \epsilon/4$. 
Hence, it holds that for every $n \ge N_2 (\epsilon)$ and every $m \ge 1$, 
\begin{equation}\label{eq:Emn-upper} 
P^x \left( E_{m,n} \ge mn \frac{\epsilon}{2} \right) \le P^x \left(  \sum_{k=1}^{mn} {\bf 1}_{\mathcal{A}_{k}^{k+n^{\prime}} \setminus \mathcal{A}_{k}^{\infty}} \ge mn \frac{\epsilon}{4} \right) \le \frac{4}{\epsilon} \sup_{x \in \mathcal{X}} P^x (n^{\prime} < T_x < \infty). 
\end{equation}

By $(U_1)$, there exists a constant $C_0$ such that 
\begin{equation}\label{eq:c0}
\sup_{x \in \mathcal{X}} P^x (n < T_x < \infty) \le C_0 (\log n)^{-1-\delta}, \ \ n \ge 1.  
\end{equation}

Thus, 
for every $\epsilon > 0$, every $n \ge N_1 (\epsilon) + N_2 (\epsilon)$ and every $m \in \mathbb{N}$, 
it follows that 
\begin{equation*}
P^x \left(R_{mn}^{j} \ge mn(F_j^{(u)} + \epsilon)\right) \le \exp\left(-C(\epsilon) m \right) + \frac{4C_0}{\epsilon} (\log n - \log \log n)^{-1-\delta}. 
\end{equation*}

Let $\eta > 0$.

For $m_k = r_k =  \lfloor (1+\eta)^{k/2} \rfloor$, 
\[ \exp\left(-C(\epsilon) m_k \right) + \frac{4C_0}{\epsilon} (\log r_k - \log \log r_k)^{-1-\delta} = O\left(k^{-1-\delta}\right), \ k \to \infty. \]
Let $a_k \coloneqq m_k r_k$. 
Then, 
 \[ \sum_{k=1}^{\infty} P^x \left(R_{a_k}^{j} \ge a_k (F_j^{(u)} + \epsilon)\right) < +\infty, \]
and by the Borel-Cantelli lemma, 
\begin{equation}\label{eq:lacunary-upper} 
\limsup_{k \to \infty} \frac{R_{a_k}^{j}}{a_k} \le F_j^{(u)} + \epsilon, \ \textup{ $P^x$-a.s.}  
\end{equation}
For $a_k \le n \le a_{k+1}$, 
$\displaystyle \frac{R^j_n}{n} \le \frac{a_{k+1}}{a_k} \frac{R_{a_{k+1}}^{j}}{a_{k+1}}$. 
By this estimate and the equality $\displaystyle \lim_{k \to \infty} \frac{a_{k+1}}{a_k} = 1+\eta$, 
\[ \limsup_{n \to \infty} \frac{R_{n}^{j}}{n} \le (1+\eta)(F_j^{(u)} + \epsilon), \ \textup{ $P^x$-a.s.}  \]
By letting  $\eta \to 0$ and $\epsilon \to 0$, 
we see that  \eqref{eq:upper-SLLN-1} holds $P^x$-a.s.

We now deal with $R^{(j)}_n$. 
In the same manner as in the derivation of \eqref{eq:lacunary-upper}, we obtain that   
\[ \limsup_{k \to \infty} \frac{R_{a_k}^{(j)}}{a_k} \le F_{(j)}^{(u)} + \epsilon, \ \textup{ $P^x$-a.s.}  \]
Since 
\[ R_n^{(j)} \le R_{a_{k+1}}^{(j)} + a_{k+1} - a_k,  \ a_k \le n \le a_{k+1},\]  
it holds that 
\[ \limsup_{n \to \infty} \frac{R_n^{(j)}}{n}  \le \limsup_{k \to \infty} \frac{a_{k+1}}{a_k} \frac{R_{a_{k+1}}^{(j)}}{a_{k+1}} + \frac{a_{k+1} - a_k}{a_k} \le (1+\eta)(F_{(j)}^{(u)} + \epsilon) + \eta, \ \textup{ $P^x$-a.s.}   \]
By letting  $\eta \to 0$ and $\epsilon \to 0$, 
we see that  \eqref{eq:upper-SLLN-2} holds $P^x$-a.s.

We now deal with the lower estimates. 

We see that 
\[ R_{mn}^{j} \ge \sum_{k=1}^{m}  (T^{j}_{k,n} - E_{m,n}), \ \textup{ and }  \  R_{mn}^{(j)} \ge \sum_{k=1}^{m} (T^{(j)}_{k,n} - E_{m,n}). \]
For ease of notation, let 
\[ F^{(l)}_{j} \coloneqq (F_{\inf})^{j-1} (1-F_{\sup}), \ \textup{ and }  \ F^{(l)}_{(j)} \coloneqq (F_{\inf})^{j-1} (1-F_{\sup})^2. \]
Then, for every $\epsilon > 0$, 
\[ P^x \left(R_{mn}^{j} \le mn(F_j^{(l)} - \epsilon)\right) \le P^x \left(  \sum_{k=1}^{m}   T^{j}_{k,n} \le mn \left(F_j^{(l)} - \frac{\epsilon}{2} \right) \right) + P^x \left(  E_{m,n} \ge n \frac{\epsilon}{2} \right). \]

We first give an upper bound for  $\displaystyle P^x \left(  \sum_{k=1}^{m}   T^{j}_{k,n} \le mn \left(F_j^{(l)} - \frac{\epsilon}{2} \right) \right)$. 
By Theorem \ref{thm:main-exp}, 
it holds that 
for every $\epsilon > 0$, 
there exists $N(\epsilon) \in \mathbb{N}$ such that 
for every $n \ge N(\epsilon)$, 
$\displaystyle \frac{E^x \left[R_n^j \right] }{n} \ge F_j^{(l)} - \frac{\epsilon}{4}$, which is equivalent with $\displaystyle \frac{E^x \left[n - R_n^j \right] }{n} \le 1 - F_j^{(l)} + \frac{\epsilon}{4}$. 

In the same manner as in the proof of Lemma \ref{lem:Laplace-estimate-multi-range}, 
we can show that 
for every $\epsilon > 0$, there exists  $N(\epsilon) \in \mathbb{N}$ such that 
for every $n \ge N(\epsilon)$, 
\[ \sup_{x \in \mathcal{X}} E^x \left[\exp\left(\frac{\theta(\epsilon/8)}{n} (n-R_n^j) \right) \right] \le \exp\left(\theta(\epsilon/8) \left(1 - F_j^{(l)} + \frac{3\epsilon}{8} \right) \right).\]

Since $T^{j}_{k,n} \le n$, 
by the Markov property, for every $\lambda > 0$ and $x \in \mathcal{X}$, 
\[  P^x \left(  \sum_{k=1}^{m}   T^{j}_{k,n} \le mn \left(F_j^{(l)} - \frac{\epsilon}{2} \right) \right) = P^x \left(  \sum_{k=1}^{m}   (n-T^{j}_{k,n}) \ge mn \left(1 - F_j^{(l)} + \frac{\epsilon}{2} \right) \right)\]
\[ \le \exp\left(-\lambda mn \left(1 - F_j^{(l)} + \frac{\epsilon}{2} \right) \right) E^x \left[ \prod_{k=1}^{m} \exp\left( \lambda (n-T^{j}_{k,n})  \right)  \right] \]
\[ \le \left( \exp\left(-\lambda n \left(1 - F_j^{(l)} + \frac{\epsilon}{2} \right) \right) \sup_{x \in \mathcal{X}} E^x \left[ \exp\left(\lambda (n-R^j_n) \right) \right] \right)^m. \]

In the same manner as in the proof of Lemma \ref{lem:upper-Markov-tail}, 
we can show that 
for every $\epsilon > 0$, there exist $C(\epsilon) > 0$  and $N_3 (\epsilon) \in \mathbb{N}$ such that 
for every $n \ge N_3 (\epsilon)$ and every $m \in \mathbb{N}$, 
\begin{equation*}
\sup_{x \in \mathcal{X}} P^x \left(  \sum_{k=1}^{m}   T^{j}_{k,n} \le mn \left(F_j^{(l)} - \frac{\epsilon}{2} \right) \right)  \le \exp\left(-C(\epsilon) m \right). 
\end{equation*}

Second, we give an upper bound for $\displaystyle P^x \left( E_{m,n} \ge n \frac{\epsilon}{2} \right)$. 
Let $n^{\prime}$ be as in \eqref{eq:def-n-prime}. 
As in \eqref{eq:Emn-upper}, by using $(U_1)$,  
we see that there exists $N_4 (\epsilon) \in \mathbb{N}$ such that for every $n \ge N_4 (\epsilon)$ and $m \in \mathbb{N}$ satisfying $mn^{\prime}/n < \epsilon/4$, 
\begin{equation*}\label{eq:Emn-upper-2} 
P^x \left( E_{m,n} \ge n \frac{\epsilon}{2} \right) \le \frac{2m}{\epsilon - 2m n^{\prime}/n}  \sup_{x \in \mathcal{X}} P^x (n^{\prime} < T_x < \infty). 
\end{equation*}

Recall \eqref{eq:c0}. 
Thus we obtain that 
for every $\epsilon > 0$, 
for every $n \ge N_3 (\epsilon) + N_4 (\epsilon)$ and every $m \in \mathbb{N}$ satisfying $mn^{\prime}/n < \epsilon/4$, 
\begin{equation*}
P^x \left(R_{mn}^{j} \le mn(F_j^{(l)} - \epsilon)\right) \le \exp\left(-C(\epsilon) m \right) + \frac{4C_0}{\epsilon} m (\log n - \log \log n)^{-1-\delta}. 
\end{equation*}

Let $\eta > 0$. 
For $m_k = \lfloor (\log k)^2 \rfloor$ and $r_k = \lfloor (1+\eta)^k \rfloor$, 
it holds that $m_k (r_k)^{\prime} / r_k < \epsilon/4$ for large $k$, and, 
\[ \exp\left(-C(\epsilon) m_k \right) + \frac{4C_0}{\epsilon} m_k (\log r_k - \log \log r_k)^{-1-\delta} = O\left(k^{-1- \frac{\delta}{2}} \right), \ k \to \infty. \] 

Let $b_k \coloneqq m_k r_k$. 
Then, 
\[ \sum_{k=1}^{\infty} P^x \left(R_{b_k}^{j} \le b_k (F_j^{(l)} - \epsilon)\right) < +\infty. \]
Hence, by the Borel-Cantelli lemma, 
\[ \liminf_{k \to \infty} \frac{R_{b_k}^{j}}{b_k} \ge F_j^{(l)} - \epsilon, \ \textup{ $P^x$-a.s.}  \]
It holds that for $b_k \le n \le b_{k+1}$, 
$\displaystyle \frac{R^j_n}{n} \ge \frac{b_{k}}{b_{k+1}} \frac{R_{b_{k}}^{j}}{b_{k}}$. 
By this estimate and the equality $\displaystyle \lim_{k \to \infty} \frac{b_{k}}{b_{k+1}} = \frac{1}{1+\eta}$, 
\begin{equation}\label{eq:lacunary-lower} 
\liminf_{n \to \infty} \frac{R_{n}^{j}}{n} \ge \frac{F_j^{(l)} - \epsilon}{1+\eta}, \ \textup{ $P^x$-a.s.}  
\end{equation}
Letting  $\eta \to 0$ and $\epsilon \to 0$, 
we see that  \eqref{eq:lower-SLLN-1} holds $P^x$-a.s.  

We now deal with $R^{(j)}_n$. 
In the same manner as in the derivation of \eqref{eq:lacunary-lower},  
\[ \liminf_{k \to \infty} \frac{R_{b_k}^{(j)}}{b_k} \ge F_{(j)}^{(l)} - \epsilon, \ \textup{ $P^x$-a.s.}  \]
Since 
\[ R_n^{(j)} \ge R_{b_{k}}^{(j)} - (b_{k+1} - b_k),  \ b_k \le n \le b_{k+1},\]  
it holds that 
\[ \liminf_{n \to \infty} \frac{R_n^{(j)}}{n}  \ge \liminf_{k \to \infty} \frac{b_{k}}{b_{k+1}} \frac{R_{b_{k}}^{(j)}}{b_{k}} - \frac{b_{k+1} - b_k}{b_k} \ge \frac{F_{(j)}^{(l)} - \epsilon}{1+\eta} - \eta, \ \textup{ $P^x$-a.s.}   \]
Letting  $\eta \to 0$ and $\epsilon \to 0$, 
we see that \eqref{eq:lower-SLLN-2} holds $P^x$-a.s.

\section{Functions of the local times}\label{sec:local-times}

Let $f$ be a non-negative function on $\mathbb{N} \cup \{0\}$ such that $f(0) = 0$. 
Let 
\[ G_n (f) \coloneqq \sum_{x \in \mathcal{X}} f(\ell(n,x)). \]
Let 
\[ C_{\inf}(f) \coloneqq \sum_{j \ge 1} f(j) (F_{\inf})^{j-1} (1-F_{\sup})^2 \]
and 
\[ C_{\sup}(f) \coloneqq \sum_{j \ge 1} f(j) (F_{\sup})^{j-1} (1-F_{\inf})^2. \]

Since $f(0) = 0$, we see that 
\[ G_n (f) = \sum_{j \ge 1} f(j) R_n^{(j)}. \]

The following theorem extends  \cite[Theorem 2]{CCMP2026} and the $L^2$-convergence part of \cite[Theorem 1]{AK2020}. 

\begin{Thm}\label{thm:func-loc-time-Lp}
Let $p \ge 1$. 
Assume that  $(U_1)$ holds and 
\[ \sum_{j \ge 1} \frac{f(j)^p}{j^{p-1}} (F_{\sup})^{j} < \infty.  \]
Then 
\begin{equation}\label{eq:L2-sup-upper} 
\lim_{n \to \infty} E^x \left[ \left(\left( \frac{G_n (f)}{n} - C_{\sup}(f) \right)_{+}\right)^p \right] = 0 
\end{equation} 
and 
\begin{equation}\label{eq:L2-inf-lower}   
\lim_{n \to \infty} E^x \left[ \left(\left( \frac{G_n (f)}{n} - C_{\inf}(f) \right)_{-}\right)^p \right] = 0. 
\end{equation} 
\end{Thm}

\begin{proof}
We remark that for every $a, b \in \mathbb{R}$, 
\[ ((a+b)_{+})^p \le (a_{+} + b_{+})^p \le 2^{p-1}((a_{+})^p + (b_{+})^p). \]

Let $f_N (j) \coloneqq f(j) {\bf 1}_{\{1,\dots, N\}}(j)$. 
Then, using $C_{\sup}(f) \ge C_{\sup}(f_N)$, 
\[ \left( \frac{G_n (f)}{n} - C_{\sup}(f) \right)_+ \le \left( \frac{G_n (f_N)}{n} - C_{\sup}(f) \right)_{+} + \frac{G_n (f) - G_n (f_N)}{n}.  \]

By \eqref{eq:exp-unif-upperbound}, 
\begin{equation}\label{eq:exp-upper-unif} 
\sup_x E^x \left[R_n^{(j)} \right] \le \sup_x E^x \left[R_n^{j} \right] \le (F_{\sup})^{j-1} \sup_x E^x \left[R_n \right] \le n (F_{\sup})^{j-1}. 
\end{equation}

Let $q$ be  the conjugate exponent of $p$, i.e., $\frac{1}{p} + \frac{1}{q} = 1$, with the convention $q=\infty$ if $p=1$. 
Using the H\"older inequality  and $n = \sum_{j \ge 1} j R_n^{(j)}$, 
\[ \sum_{j \ge N+1} f(j) R_n^{(j)} = \sum_{j \ge N+1} \frac{f(j)}{j} (j R_n^{(j)})^{1/p} \cdot (j R_n^{(j)})^{1/q} \le n^{1/q} \left(\sum_{j \ge N+1} \frac{f(j)^p}{j^{p-1}} R_n^{(j)}\right)^{1/p}. \]
Hence, 
\[ \left(  \frac{G_n (f)}{n} -  \frac{G_n (f_N)}{n} \right)^p  \le \sum_{j \ge N+1} \frac{f(j)^p}{j^{p-1}} \frac{R_n^{(j)}}{n}.  \]
Using this and \eqref{eq:exp-upper-unif}, 
\[ E^x \left[ \left(  \frac{G_n (f)}{n} -  \frac{G_n (f_N)}{n} \right)^p \right] \le \sum_{j \ge N+1} \frac{f(j)^p}{j^{p-1}}(F_{\sup})^{j-1}.  \]

Since $ \frac{G_n (f_N)}{n} \le \max_{1 \le j \le N} f(j)$, 
\eqref{eq:upper-SLLN-2} and the Lebesgue convergence theorem yield that 
\[ \lim_{n \to \infty} E^x \left[ \left( \left( \frac{G_n (f_N)}{n} -  C_{\sup}(f_N)  \right)_{+} \right)^p \right] = 0. \]
Using this and $  C_{\sup}(f_N) \le   C_{\sup}(f)$, we obtain that 
\[ \lim_{n \to \infty} E^x \left[ \left( \left( \frac{G_n (f_N)}{n} -  C_{\sup}(f)  \right)_{+} \right)^p \right] = 0. \]

Hence, 
\[  \limsup_{n \to \infty} E^x \left[ \left(\left( \frac{G_n (f)}{n} - C_{\sup}(f) \right)_{+}\right)^p \right] \le 2^{p-1} \sum_{j \ge N+1} \frac{f(j)^p}{j^{p-1}}(F_{\sup})^{j-1}. \]
Letting $N \to \infty$, we have \eqref{eq:L2-sup-upper}. 
The same argument gives  \eqref{eq:L2-inf-lower}.  
\end{proof}

\begin{Prop}
Assume that $(U_1)$ holds and $F_{\inf} < 1$. 
Then for every $x \in \mathcal{X}$, 
\begin{equation}\label{eq:inf-sup-sandwich} 
C_{\inf}(f) \le \liminf_{n \to \infty} \frac{G_n (f)}{n} \le C_{\sup}(f), \ \textup{ $P^x$-a.s.}  
\end{equation}
\end{Prop}

\begin{proof}
The lower bound follows from the non-negativity of $f$ and Theorem \ref{thm:main-SLLN}. 
If $C_{\sup}(f) = \infty$, the upper bound is obvious. 
Assume that $C_{\sup}(f) < \infty$. 
By Theorem \ref{thm:func-loc-time-Lp} for $p=1$, 
there exists a sequence $(n_k)_k$ such that 
\[ \lim_{k \to \infty}  \left( \frac{G_{n_k} (f)}{n_k} - C_{\sup}(f) \right)_{+} = 0, \ \textup{ $P^x$-a.s.}  \]
This is equivalent to  
\[ \limsup_{k \to \infty}  \frac{G_{n_k} (f)}{n_k} \le C_{\sup}(f), \ \textup{ $P^x$-a.s.} \]
This implies the upper bound. 
\end{proof}

If $f(j) = j$ and $F_{\inf} = 1$, then $G_n (f) = n$ and $C_{\sup}(f) = 0$ so that \eqref{eq:inf-sup-sandwich} fails.

The following lower bound on the expectation plays a role similar to that of  \cite[Lemma 3]{CCMP2026}. 

\begin{Lem}\label{lem:exp-lower-quantitative}
Let $x \in \mathcal{X}, n \ge 1$. 
Then, \\
(i) $E^x \left[ R_n^{(1)} \right] \ge (1-F_{\sup})^2 n$.\\
(ii) If $j \ge 2$, then for every  $\ell \in \{1,\dots, n\}$, 
\[ E^x \left[ R_n^{(j)} \right] \ge (1-F_{\sup})^2 \ell F_{\inf}\left( \left\lfloor \frac{n-\ell}{j-1} \right\rfloor\right)^{j-1}. \]
\end{Lem}

\begin{proof}
For $y \in \mathcal{X}$, 
let $ T_y^{(1)} \coloneqq T_y$ and 
\[ T_y^{(j)} \coloneqq \inf \left\{n > T_y^{(j-1)} \colon X_n = y \right\}, \ j \ge 2. \]

By the Markov property, we obtain that 
\begin{equation}\label{eq:multi-ineq-Markov}
P^y \left( T_y^{(j)} \le n \right) \ge P^y \left( T_y \le \left\lfloor \frac{n}{j} \right\rfloor \right)^{j}, \ j \ge 1, n \ge 1, y \in \mathcal{X}. 
\end{equation}
We also see that 
\begin{equation*}
E^x \left[ R_n^j \right] = \sum_{y \in \mathcal{X}}  P^x \left( T_y^{(j)} \le n \right). 
\end{equation*}
By this and the Markov property, 
\[ E^x \left[ R_n^j \right] = \sum_{y \in \mathcal{X}}  \sum_{k = 1}^{n}  P^x \left( T_y^{(j)} \le n, T_y = k \right) = \sum_{y \in \mathcal{X}}  \sum_{k = 1}^{n}  P^x \left( T_y = k \right) P^y \left( T_y^{(j-1)} \le n - k \right). \]
By this and \eqref{eq:multi-ineq-Markov}, 
\[  E^x \left[ R_n^j \right] \ge \sum_{y \in \mathcal{X}} P^x \left( T_y \le \ell \right) P^y \left( T_y^{(j-1)} \le n - \ell \right) \ge E^x \left[R_{\ell}\right] F_{\inf}\left( \left\lfloor \frac{n-\ell}{j-1} \right\rfloor\right)^{j-1}. \]
Applying \eqref{eq:comparision-ineq-1} and  \eqref{eq:exp-lower-most-basic}, we obtain the assertion. 
\end{proof}

We now consider the following stronger uniform tail  condition $(U_2)$: 
\[ \sup_{x} P^x (n < T_x < \infty) = O(n^{-\delta}), \ n \to \infty, \delta > 0. \]
This condition was considered in \cite[Proposition 5.7]{KN2018}. 
The simple random walk on $\mathbb{Z}^d, d \ge 3$, satisfies this condition. 

The following corresponds to \cite[Remark 3]{CCMP2026}. 
\begin{Prop}\label{prop:Lp-explode}
Assume that $(U_2)$ holds, that $0 < F_{\inf} \le F_{\sup} < 1$, and that $p \ge 1$. 
Let  $\alpha > p-2 + \frac{p-1}{\delta}$. 
Let $f(0) \coloneqq 0$ and $f(j) \coloneqq j^{\alpha/p} F_{\inf}^{- j/p}$ for $j \ge 1$. 
Then 
\[ \lim_{n \to \infty} E^x \left[ \left( \frac{G_n (f)}{n} \right)^p \right] = \infty. \]
\end{Prop}

\begin{proof}
We first remark that 
\[ E^x \left[ \left( \frac{G_n (f)}{n} \right)^p \right] \ge \frac{1}{n^p} \sum_{j \ge 1} f(j)^p E^x \left[\left(R_n^{(j)}\right)^p\right] \ge \frac{1}{n^p} \sum_{j \ge 2} f(j)^p E^x \left[R_n^{(j)}\right]. \]
By Lemma \ref{lem:exp-lower-quantitative}, 
\[ \frac{1}{n^p} \sum_{j \ge 2} f(j)^p E^x \left[R_n^{(j)}\right] \ge \frac{(1-F_{\sup})^2}{n^p}   \left\lfloor \frac{n}{2} \right\rfloor \sum_{j \ge 2}  j^{\alpha} F_{\inf}^{-j} F_{\inf}\left( \left\lfloor \frac{n}{2(j-1)} \right\rfloor\right)^{j-1}. \]
By  $(U_2)$, 
there exists $C > 0$ such that for every $n \ge 1$ and every $j \ge 2$, 
\[ F_{\inf}\left( \left\lfloor \frac{n}{2(j-1)} \right\rfloor\right) \ge  \left(1 - C \left(\frac{j-1}{n}\right)^{\delta}\right)_{+}  F_{\inf}. \]
Hence it suffices to show that 
\begin{equation}\label{eq:lim-infty-sum-j}
\lim_{n \to \infty} \frac{1}{n^{p-1}} \sum_{j \ge 2} j^{\alpha} \left(\left(1 - C \left(\frac{j-1}{n}\right)^{\delta}\right)_{+}\right)^{j-1} = \infty. 
\end{equation}

Let $a > 0$ such that $C a^{\delta + 1} < 1/2$. 
If $j - 1 \le \lfloor a n^{\delta/(1+\delta)} \rfloor$, then 
\[ \left(1 - C \left(\frac{j-1}{n}\right)^{\delta}\right)^{j-1}  \ge 1 - C \frac{(j-1)^{1+\delta}}{n^{\delta}} \ge \frac{1}{2}. \]
Hence
\[ \sum_{j \ge 2} j^{\alpha} \left(\left(1 - C \left(\frac{j-1}{n}\right)^{\delta}\right)_{+}\right)^{j-1} \ge \frac{1}{2} \sum_{j=2}^{\lfloor a n^{\delta/(1+\delta)} \rfloor + 1} j^{\alpha}  \]
and the right-hand side grows on the scale of $n^{(\alpha + 1) \delta / (\delta + 1)}$. 
By the assumption, $(\alpha + 1) \delta >  (p-1) (\delta + 1)$ and we obtain \eqref{eq:lim-infty-sum-j}. 
\end{proof}

\begin{Rem}
Suppose that the assumptions of Proposition \ref{prop:Lp-explode} hold. 
If, additionally, $p > 1$ and $F_{\sup} < F_{\inf}^{1/p}$, 
then $C_{\sup}(f) < \infty$ and 
\[ \lim_{n \to \infty} E^x \left[ \left( \frac{G_n (f)}{n} - C_{\sup}(f) \right)_{+} \right] =  \lim_{n \to \infty} E^x \left[ \left( \frac{G_n (f)}{n} - C_{\inf}(f) \right)_{-} \right] = 0. \] 
If $F_{\inf} = F_{\sup} < 1$ and $p >1$,  then $F_{\sup} < F_{\inf}^{1/p}$. 
\end{Rem}

In general, the strict inequality $F_{\inf} < F_{\sup}$ can occur. 
See \cite[Theorem 1.3]{Okamura2014}. 
By the same argument as in the proof of  \cite[Theorem 1.3]{Okamura2014}, 
one can show that 
\[ (F_{\inf})^{j-1} (1-F_{\sup}) =  \liminf_{n \to \infty} \frac{\inf_{x \in \mathcal{X}} E^{x} \left[R_n^j \right]}{n} \]
\[ < \limsup_{n \to \infty} \frac{\sup_{x \in \mathcal{X}} E^{x} \left[R_n^j \right]}{n} = (F_{\sup})^{j-1} (1-F_{\inf}) \]
for each $j$.

\section{Examples}\label{sec:examples}

We give a sufficient condition for $(U_1)$. 
Let $p_n (x,y) \coloneqq P^x (X_n = y)$. 
Then we have the trivial estimate 
$P^x (n < T_x < \infty) \le \sum_{k=n+1}^{\infty} p_k (x,x)$. 

Specifically, if 
there exists a constant  $C_0$ such that
\begin{equation}\label{eq:Nash-type}
p_n (x,x) \le \frac{C_0}{n (\log (n+1))^{2+\delta}} 
\end{equation} 
holds for every $x \in \mathcal{X}, n \ge 1$, 
then $(U_1)$ holds. 

\begin{Prop}
Assume that $(X_n)_n$ is the simple random walk on an infinite connected simple graph with bounded degrees. 
Then the Nash-type inequality \eqref{eq:Nash-type}  is stable under rough isometries between graphs. 
\end{Prop}

\begin{proof}
We apply results of Tessera \cite{Tessera2008}. 
Let $\gamma(t) \coloneqq C_0 t^{-1} (\log (t+1))^{-2-\delta}$. 
Then there exists an increasing positive function $\varphi$ such that 
$t = \int_0^{1/\gamma(t)} \varphi(v)^2 \frac{dv}{v}$ for every $t \ge 1$. 
Moreover, 
$\varphi(v) \sim \frac{\sqrt{C_0 v}}{\log(v+1)^{1+\delta/2}}$ as $v \to \infty$. 

Let $\mathcal{X}_1$ and $\mathcal{X}_2$ be two infinite connected simple graphs with bounded degrees. 
Assume that \eqref{eq:Nash-type} holds for the simple random walk on $\mathcal{X}_1$. 
By \cite[Theorem 3.5 (ii)]{Tessera2008}, 
the Sobolev inequality associated with $\varphi$ holds for $\mathcal{X}_1$. 
By \cite[Theorem 8.1]{Tessera2008}, 
the Sobolev inequality is stable under rough isometries, and hence, 
the Sobolev inequality associated with $\varphi$ also holds for $\mathcal{X}_2$. 
By \cite[Theorem 3.5 (i)]{Tessera2008}, 
 \eqref{eq:Nash-type} also holds for the simple random walk on $\mathcal{X}_2$. 
\end{proof}

The Laakso-type graph constructed in Murugan \cite[Theorem 5.13]{Murugan2025} is an infinite connected simple graph with bounded degrees such that 
there exist two constants  $C_1$ and $C_2$  such that the inequalities 
\begin{equation}\label{eq:Nash-type-both}
\frac{C_1}{n (\log (n+1))^{2+\delta}}  \le p_n (x,x) + p_{n+1}(x,x) \le \frac{C_2}{n (\log (n+1))^{2+\delta}} 
\end{equation} 
hold for every $x \in \mathcal{X}, n \ge 1$. 

Every infinite, connected, locally finite, vertex-transitive graph with polynomial volume growth does not satisfy \eqref{eq:Nash-type-both}, 
because every such graph is roughly isometric to a Cayley graph of a virtually nilpotent group by Trofimov \cite{Trifomov1985} and the claim follows from Hebisch and Saloff-Coste \cite{HSC1993}. 

There are examples of random walks on $\mathbb{Z}$ with long-range jumps satisfying \eqref{eq:Nash-type-both}. 
By Murugan and Saloff-Coste \cite[Theorem 1.1]{MSC2015}, 
if there exist two constants  $C_3$ and $C_4$  such that the inequalities 
\[  C_3 \frac{(\log(e+|x-y|))^{2+\delta}}{(1+|x-y|)^2} \le p(x,y) = p(y,x) \le C_4 \frac{(\log(e+|x-y|))^{2+\delta}}{(1+|x-y|)^2}\]
hold for every $x, y \in \mathcal{X}$, 
then \eqref{eq:Nash-type-both} holds. \\

\noindent{\it Acknowledgments.} \ The author is supported by JSPS KAKENHI JP22K13928. \\

\bibliographystyle{plain}
\bibliography{MPR}

\begin{thebibliography}{10}

\bibitem{AK2020}
Inna~M. Asymont and Dmitry Korshunov.
\newblock Strong law of large numbers for a function of the local times of a
  transient random walk in {{\({\mathbb{Z}}^d\)}}.
\newblock {\em J. Theor. Probab.}, 33(4):2315--2336, 2020.

\bibitem{BK2009}
Mathias Becker and Wolfgang K\"onig.
\newblock Moments and distribution of the local times of a transient random
  walk on {$\Bbb Z^d$}.
\newblock {\em J. Theoret. Probab.}, 22(2):365--374, 2009.

\bibitem{CCMP2026}
Yinshan Chang, Qinwei Chen, Qian Meng, and Xue Peng.
\newblock Strong law of large numbers for a function of the local time of a
  transient random walk on a group.
\newblock {\em J. Theor. Probab.}, 39(1):16, 2026.
\newblock Id/No 7.

\bibitem{Derriennic1980}
Yves Derriennic.
\newblock Quelques applications du th\'eor\`eme ergodique sous-additif.
\newblock In {\em Conference on {R}andom {W}alks ({K}leebach, 1979)
  ({F}rench)}, volume~74 of {\em Ast\'erisque}, pages 183--201. Soc. Math.
  France, Paris, 1980.

\bibitem{ET1960}
P.~Erd\H{o}s and S.~J. Taylor.
\newblock Some problems concerning the structure of random walk paths.
\newblock {\em Acta Math. Acad. Sci. Hungar.}, 11:137--162, 1960.

\bibitem{Flatto1976}
Leopold Flatto.
\newblock The multiple range of two-dimensional recurrent walk.
\newblock {\em Ann. Probability}, 4(2):229--248, 1976.

\bibitem{hamana1997}
Yuji Hamana.
\newblock The fluctuation result for the multiple point range of two
  dimensional recurrent random walks.
\newblock {\em Ann. Probab.}, 25(2):598--639, 1997.

\bibitem{Hamana1998}
Yuji Hamana.
\newblock A remark on the multiple point range of two-dimensional random walks.
\newblock {\em Kyushu J. Math.}, 52(1):23--80, 1998.

\bibitem{HSC1993}
W.~Hebisch and L.~Saloff-Coste.
\newblock Gaussian estimates for {Markov} chains and random walks on groups.
\newblock {\em Ann. Probab.}, 21(2):673--709, 1993.

\bibitem{KN2018}
Takashi Kumagai and Chikara Nakamura.
\newblock Lamplighter random walks on fractals.
\newblock {\em J. Theor. Probab.}, 31(1):68--92, 2018.

\bibitem{Murugan2025}
Mathav Murugan.
\newblock Diffusions and random walks with prescribed sub-{G}aussian heat
  kernel estimates.
\newblock {\em Ann. Probab.}, to appear, 2025.

\bibitem{MSC2015}
Mathav Murugan and Laurent Saloff-Coste.
\newblock Transition probability estimates for long range random walks.
\newblock {\em New York J. Math.}, 21:723--757, 2015.

\bibitem{Okamura2014}
Kazuki Okamura.
\newblock On the range of random walk on graphs satisfying a uniform condition.
\newblock {\em ALEA, Lat. Am. J. Probab. Math. Stat.}, 11(2):341--357, 2014.

\bibitem{Pitt1974}
Joel~H. Pitt.
\newblock Multiple points of transient random walks.
\newblock {\em Proc. Am. Math. Soc.}, 43:195--199, 1974.

\bibitem{SCZ2016}
Laurent Saloff-Coste and Tianyi Zheng.
\newblock Random walks and isoperimetric profiles under moment conditions.
\newblock {\em Ann. Probab.}, 44(6):4133--4183, 2016.

\bibitem{Tessera2008}
Romain Tessera.
\newblock Large scale {S}obolev inequalities on metric measure spaces and
  applications.
\newblock {\em Rev. Mat. Iberoam.}, 24(3):825--864, 2008.

\bibitem{Trifomov1985}
V.~I. Trofimov.
\newblock Graphs with polynomial growth.
\newblock {\em Math. USSR, Sb.}, 51:405--417, 1985.

\bibitem{Woess2013}
Wolfgang Woess.
\newblock What is a horocyclic product, and how is it related to lamplighters?
\newblock {\em Int. Math. Nachr., Wien}, 224:1--27, 2013.

\end{thebibliography}
 
\end{document}